\documentclass[12pt,reqno]{amsart}
\usepackage[left=80pt,right=80pt]{geometry}
\usepackage[usenames]{color}
\usepackage{amsmath}
\usepackage{amssymb}
\usepackage{amsthm}
\usepackage{enumitem}
\usepackage{float}
\usepackage{graphicx}
\usepackage{tikz}
\usepackage{cases}
\usepackage{blkarray}
\usepackage{tikz-qtree}
\usetikzlibrary{graphs,graphs.standard,calc}
\usepackage[caption=false]{subfig}
\usepackage{hyperref,url}
\hypersetup{
	colorlinks=true,
  linkcolor=black,          % color of internal links (change box color with linkbordercolor)
  citecolor=black,         % color of links to bibliography
  filecolor=black,      % color of file links
  urlcolor=black           % color of external links
}

\newtheorem{prop}{Proposition}

\newtheorem{theorem}[prop]{Theorem}

\newtheorem{corollary}[prop]{Corollary}
\theoremstyle{definition}

\newtheorem{remark}[prop]{Remark}
\newtheorem{example}[prop]{Example}
\newtheorem{question}[prop]{Question}
\newtheorem{observation}[prop]{Observation}
\newtheorem{construction}[prop]{Construction}

\newcommand{\seqnum}[1]{\href{https://oeis.org/#1}{\rm \underline{#1}}}
\allowdisplaybreaks
\newcommand{\mylabel}[2]{#2\def\@currentlabel{#2}\label{#1}}

\begin{document}
\tikzset{mystyle/.style={matrix of nodes,
        nodes in empty cells,
        row 1/.style={nodes={draw=none}},
        row sep=-\pgflinewidth,
        column sep=-\pgflinewidth,
        nodes={draw,minimum width=1cm,minimum height=1cm,anchor=center}}}
\tikzset{mystyleb/.style={matrix of nodes,
        nodes in empty cells,
        row sep=-\pgflinewidth,
        column sep=-\pgflinewidth,
        nodes={draw,minimum width=1cm,minimum height=1cm,anchor=center}}}

\title{On the maximal sum of the entries of a matrix power}

\author[SELA FRIED]{Sela Fried$^{\dagger}$}
%\thanks{$^{*}$ Corresponding author.}
\thanks{$^{\dagger}$ Department of Computer Science, Israel Academic College,
52275 Ramat Gan, Israel.
\\
\href{mailto:friedsela@gmail.com}{\tt friedsela@gmail.com}}
\author[TOUFIK MANSOUR]{Toufik Mansour$^{\sharp}$}
\thanks{$^{\sharp}$ Department of Mathematics, University of Haifa, 3103301 Haifa,
Israel.\\
\href{mailto:tmansour@univ.haifa.ac.il}{\tt tmansour@univ.haifa.ac.il}}

\iffalse
\author{Sela Fried}
\address{Department of Computer Science, Israel Academic College,
52275 Ramat Gan, Israel}
\email{friedsela@gmail.com}

\author{Toufik Mansour}
\address{Department of Mathematics, University of Haifa, 3498838 Haifa,
Israel}
\email{tmansour@univ.haifa.ac.il}
\fi
\maketitle

\begin{abstract}
Let $p_n$ be the maximal sum of the entries of $A^2$, where $A$ is a square matrix of size $n$, consisting of the numbers $1,2,\ldots,n^2$, each appearing exactly once. We prove that $m_n=\Theta(n^7)$. More precisely, we show that $n(240n^{6}+28n^{5}+364n^{4}+210n^{2}-28n+26-105((-1)^{n}+1))/840\leq p_n\leq n^{3}(n^{2}+1)(7n^{2}+5)/24$.
\bigskip


\smallskip

\noindent
\textbf{Math.~Subj.~Class.:} 05B20, 15A15, 15A45.
\end{abstract}

\section{Introduction}
This work is concerned with the following question.
\begin{question}\label{q;1}
Let $A$ be a square matrix of size $n$, consisting of the numbers $1,2,\ldots,n^2$, each appearing exactly once. Let $m\geq 2$ be a natural number. What is the maximal sum of the entries of $A^m$?
\end{question}
This question was inspired by the work of Gasper et al.\ \cite{Sigg1}, who were interested in the maximal determinant of $A$ (see also \seqnum{A085000} in the On-Line Encyclopedia of Integer Sequences \cite{OL}). Following their approach, we relax the problem by considering arbitrary real square matrices with a prescribed sum and sum of squares of their entries. The method of Lagrange multipliers allows us then to obtain an upper bound on the maximal sum of entries, for $m=2$. The lower bound is obtained by constructing special matrices that, based on empirical evidence, have favorable properties in this regard.

\section{Main results}
Let $n$ be a natural number to be used throughout this work. All matrices are tacitly assumed to be square of size $n$. We shall make use of the following notation. For a matrix $A$ and $1\leq i,j\leq n$, we denote by $(A)_{ij}$ the $ij$th entry of $A$. We denote by $s(A)$ the sum of the entries of $A$ and by $q(A)$ the sum of their squares, i.e., $s(A)=\sum_{i,j=1}^n (A)_{ij}$ and $q(A)=\sum_{i,j=1}^n (A)^2_{ij}$. The transpose of $A$ is denoted by $A^T$. We denote by $p_n$ the answer to Question \ref{q;1}, for $m=2$.

\subsection{The upper bound}

The results we obtain in this part are based on the following observation.

\begin{observation}
Let $m\geq 2$ be a natural number and let $X$ be a matrix whose entries are variables. We define the Lagrange function by $$L(X,\lambda,\mu)=s(X^m)-\lambda(s(X)-s(A))-\mu(q(X)-q(A)).$$ It is well known (e.g. \cite[(90)]{Pet}) that, for every $1\leq k,\ell,s,t\leq n$, we have $$\frac{\partial(X^{m})_{k\ell}}{\partial (X)_{st}}=\sum_{r=0}^{m-1}\left(X^{r}J^{st}X^{m-1-r}\right)_{k\ell},$$ where $J^{st}$ is the matrix having $1$ at the $st$th entry and $0$ elsewhere. By the Lagrange multipliers theorem, there exist real $\lambda,\mu$ such that, for every $1\leq s,t\leq n$, we have $\partial L(X,\lambda,\mu)/\partial (X)_{st}=0$, i.e.,
\begin{align}
&\sum_{k,\ell=1}^{n}\sum_{r=0}^{m-1}\left(X^{r}J^{st}X^{m-1-r}\right)_{k\ell}-\lambda-2\mu (X)_{st}=0\iff\nonumber\\&\sum_{r=0}^{m-1}\left(\sum_{\ell=1}^{n}\left(X^{m-1-r}\right)_{t\ell}\right)\left(\sum_{k=1}^{n}(X^{r})_{ks}\right)-\lambda-2\mu (X)_{st}=0.\label{eq;a1}\end{align}
Furthermore, $s(X)=s(A)$ and $q(X)=q(A)$.
Equation (\ref{eq;a1}), taken over $1\leq k,\ell\leq n$, may be written compactly in matrix form as
\begin{equation}\label{eq;a2}\sum_{r=0}^{m-1}\left(J(X^{T})^{m-1-r}\right)\circ\left((X^{T})^{r}J\right)=\lambda J+2\mu X,\end{equation} where $J$ is the matrix all of whose entries are $1$ and $\circ$ stands for the Hadamard product. Equation (\ref{eq;a2}) seems to be hard to analyze, for arbitrary $m$. Nevertheless, for $m=2$, we obtain the following result.
\end{observation}

%\subsection{An upper bound}

\begin{theorem}
Let $A$ be a matrix. Then $$s(A^2) \leq \frac{s(A)^{2}}{n}+\frac{n}{2}\left|q(A)-\frac{s(A)^{2}}{n^{2}}\right|.$$
\end{theorem}

\begin{proof}
For $m=2$ and $1\leq s,t\leq n$, equations (\ref{eq;a1}) and (\ref{eq;a2}) take the forms \begin{equation}\label{eq;1}\sum_{\ell=1}^{n}(X)_{t\ell}+\sum_{k=1}^{n}(X)_{ks}-\lambda-2\mu (X)_{st}=0\end{equation} and \begin{equation}\label{eq;104}
 JX^{T}+X^{T}J=\lambda J+2\mu X,
\end{equation} respectively. Summing (\ref{eq;1}) over $1\leq s,t\leq n$, we obtain \begin{equation}\label{eq;10}
    \lambda=\frac{2s(A)(n-\mu)}{n^{2}}.
\end{equation} Multiplying (\ref{eq;1}) by $(X)_{st}$, summing over $1\leq s,t\leq n$ and using (\ref{eq;10}), we obtain
\begin{equation}\label{eq;3}s(X^{2})=\frac{s(A)^{2}}{n}+\mu\left(q(A)-\frac{s(A)^{2}}{n^{2}}\right).\end{equation}
If $q(A)-s(A)^2/n^2=0$, we are done. Assuming otherwise, it suffices to bound $\mu$. To this end, let $Y = X^T-X$. Transposing (\ref{eq;104}) and subtracting the result from (\ref{eq;104}), we obtain the homogeneous Sylvester equation
\begin{equation}\label{eq;105}
(J+2\mu I)Y+YJ=0,
\end{equation} which has only the trivial solution if and only if the matrices $J+2\mu I$ and $J$ have disjoint spectra (cf.\ \cite{Syl}). It is easy to see (e.g., \cite[proof of Lemma 2.1]{Sigg1}) that these  are $\{2\mu,2\mu+n\}$ and $\{0,n\}$, respectively. Thus, the spectra are disjoint if and only if $\mu\neq 0,\pm n/2$. Assume that this is the case. Then only $Y=0$ solves (\ref{eq;105}) and therefore $X^T=X$. Thus, (\ref{eq;104}) may be rewritten as
\begin{equation}\label{eq;206}(2\mu I-J)X=XJ-\lambda J.\end{equation} The matrix $2\mu I-J$ is invertible if and only if $\mu\neq 0,n/2$ and, in this case (e.g., \cite[Lemma 2.1]{Sigg1}), $$(2\mu I-J)^{-1}=\frac{1}{2\mu}I+\frac{1}{2\mu(2\mu-n)}J.$$ Multiplying (\ref{eq;206}) by $(2\mu I-J)^{-1}$, we conclude that

$$X=\frac{1}{2\mu-n}(XJ-\lambda J).$$ All the columns of the matrix $XJ-\lambda J$ are equal and therefore also all the columns of $X$, which is symmetric. Thus, all the entries of $X$ are equal. Since $s(X)=s(A)$, necessarily $X=(s(A)/n^2)J$. It follows that $q(X)=s(A)^2/n^2$, a contradiction. We conclude that $|\mu|\leq n/2$ and the assertion follows from (\ref{eq;3}).
\end{proof}

\begin{corollary}
We have $$p_n\leq \frac{n^{3}(n^{2}+1)(7n^{2}+5)}{24}.$$
\end{corollary}

\subsection{The lower bound}
Empirically, we find that $p_1 = 1, p_2 = 54, p_3 = 761, p_4 \geq 5284, p_5 \geq 24303, p_6 \geq 85352$ and $p_7\geq 248045$. The lower bounds were obtained with the help of a simple hill-climbing algorithm (e.g., \cite[4.1.1]{R}). % that successively interchanges two random elements of some initial matrix $A$ if it increases the value of $s(A^2)$. Some exploration is needed to avoid local minima. and convergence seems to be fastest when the exploration rate is $\sim 0.001$.
Here are some matrices that achieve the values mentioned above.
\begin{align*}
n=4&\colon\begin{pmatrix}
1& 2& 7& 6\\
 3& 4& 12& 9\\
 8& 11&  16 & 15\\
 5 & 10& 14 & 13
\end{pmatrix}
&n=5&\colon\begin{pmatrix}
25& 23& 17& 13& 21\\
 24& 22& 11& 10& 18\\
 16& 12&  4 & 2&  8\\
 14 & 9&  3 & 1&  5\\
 20& 19&  7&  6& 15
\end{pmatrix}\\
n=6&\colon\begin{pmatrix}
11 &26&  5&  8& 28& 16\\
 25& 33& 14& 21& 35 &29\\
  6& 15&  1&  2& 18& 10\\
  7& 20&  3&  4& 22& 13\\
 27 &34& 19& 23& 36& 32\\
 17& 30&  9& 12& 31& 24
\end{pmatrix}
&n=7&\colon\begin{pmatrix}
1 &15  &2& 24  &9 & 6 &20\\
 14 &41& 18& 45 &32& 27 &43\\
  3 &19&  4& 29 &11 & 8 &25\\
 23 &44& 30& 49 &40 &36 &47\\
 10 &31& 12& 39& 22& 16& 38\\
  5 &28&  7 &35& 17 &13 &34\\
 21 &42& 26 &48& 37 &33 &46\\
 \end{pmatrix}
\end{align*}

Based on our experiments, we believe that if $A$ is such that $s(A^2)=p_n$, then it is necessary that:
\begin{enumerate}
    \item [(a)] For every $1\leq i,j\leq n$ we have $|(A)_{ij}-(A)_{ji}|\leq 1$.
    \item [(b)] If $n$ is odd, then, for every $1\leq i\leq n$, the sum of the elements of the $i$th row is equal to the sum of the elements of the $i$th column.
    \item [(c)] If $n$ is even, then, for every $1\leq i\leq n$, the sum of the elements of the $i$th row differs by exactly $1$ from the sum of the elements of the $i$th column. For half of the rows, the difference is negative, and for half of the rows, the difference is positive. The same holds for the columns.
    \item [(d)] Depending on $n$, certain numbers must lie on the main diagonal. These always include $1$ and $n^2$.
\end{enumerate}

\begin{remark}
    Of course, taking $\mu=-n/2$ in equation (\ref{eq;3}) gives a lower bound on $p_n$. A better one is based on the following construction.
\end{remark}

\begin{construction}\label{con;1}
We inductively construct matrices $A_n$ that have large $s(A_n^2)$. Set $A_1 = (1)$ and assume we have already constructed $A_{n-1}$. We define $A_n$ in two steps. First, we let
 \[A'_n =
  \begin{blockarray}{cccccc}
  \begin{block}{(cccccc)}
    n^2  & n^2-1 & n^2-3 & \cdots &  & (n-1)^2 +2 \\
  \cline{2-6}
  \begin{block*}{c|ccccc}
    n^2-2 &  &  &  &  &\\
    n^2-4   &   &  &   &  & \\
     \vdots  &   &  & A_{n-1}  &   &\\
       &  &  &  &  &\\
    (n-1)^2 +1    &  &  & &  &\\
  \end{block*}
  \end{block}
  \end{blockarray}
\]
Then, we define $A_n$ to be the matrix obtained from $A'_n$ by interchanging $(A'_n)_{1k}$ and $(A'_n)_{k1}$ if $k$ is odd, for every $2\leq k \leq n$.
\end{construction}

\begin{example}
$$A_7=\begin{pmatrix}
    49&48&45&44&41&40&37\\
 47&36&35&32&31&28&27\\
 46&34&25&24&21&20&17\\
 43&33&23&16&15&12&11\\
 42&30&22&14& 9& 8& 5\\
 39&29&19&13& 7& 4& 3\\
 38&26&18&10& 6& 2& 1
 \end{pmatrix}$$
\end{example}

\begin{theorem}
The matrix $A_n$ from Construction \ref{con;1} satisfies conditions (a) - (d) above and we have $$s(A_n^2)=\frac{n}{840}\left(240n^{6}+28n^{5}+364n^{4}+210n^{2}-28n+26-105((-1)^{n}+1)\right).$$
\end{theorem}
\begin{proof}
Conditions (a) and (d) are clear.
For a matrix $B$ and $1\leq k\leq n$, denote by $R_k(B)$ (resp.\ $C_k(B)$) the sum of the $k$th row (resp.\ column) of $B$. It is immediately verified that
\begin{equation}\label{eq;123}
s(B^2)=\sum_{k=1}^nR_k(B)C_k(B).
\end{equation}
It is not hard to see that, for every $1\leq k\leq n$, we have \begin{align*}
R_{n-k+1}(A'_n) &=\sum_{j=1}^{n-k}\left(k^{2}+\sum_{\ell=1}^{j}(2k-1+2(\ell-1))\right)+k^{2}+\sum_{j=1}^{k-1}\left(k^{2}-1-2(j-1)\right),\\
C_{n-k+1}(A'_n) &=\sum_{j=1}^{n-k}\left(1+k^{2}+\sum_{\ell=1}^{j}(2k-1+2(\ell-1))\right)+k^{2}+\sum_{j=1}^{k-1}\left(k^{2}-2-2(j-1)\right).
\end{align*} Furthermore, if $n$ is odd then
\begin{align}
R_{n-k+1}(A_n) &= R_{n-k+1}(A'_n)+\frac{n+1-2k}{2},\label{aa1}\\
C_{n-k+1}(A_n) &= C_{n-k+1}(A'_n)-\frac{n+1-2k}{2} \label{aa2}
\end{align} and, if $n$ is even, then
\begin{align}
R_{n-k+1}(A_n) &= R_{n-k+1}(A'_n)+\frac{n+1-2k+(-1)^k}{2},\label{aa3}\\
C_{n-k+1}(A_n) &= C_{n-k+1}(A'_n)-\frac{n+1-2k+(-1)^k}{2}.\label{aa4}
\end{align}
From this, it is easy to see that conditions (b) and (c) are also satisfied. Finally, multiplying (\ref{aa1}) and (\ref{aa2}) (or (\ref{aa3}) and (\ref{aa4})), summing over $1\leq k\leq n$ and using (\ref{eq;123}) and Faulhaber's formula (e.g.\ \cite{B}), we conclude that
$$s(A_n^2)=
\begin{cases}
\frac{n}{420}\left(120n^6 + 14n^5 + 182n^4 + 105n^2 - 14n + 13\right) &\textnormal{if } n \textnormal{ is odd}\\
\frac{n}{420}\left(120n^6 + 14n^5 + 182n^4 + 105n^2 - 14n - 92\right) &\textnormal{if } n \textnormal{ is even},
\end{cases}$$ from which the last assertion immediately follows.
\end{proof}

\begin{corollary}
We have $$p_n\geq \frac{n}{840}\left(240n^{6}+28n^{5}+364n^{4}+210n^{2}-28n+26-105((-1)^{n}+1)\right).$$
\end{corollary}

%\paragraph{\textbf{Acknowledgements}}
%This research received no specific grant from any funding agency in the public, commercial, or not-for-profit sectors.

\end{document}